\documentclass[12pt,reqno]{amsart}

\usepackage{amssymb,latexsym}

\usepackage{enumerate}

\usepackage[french,english]{babel}
\usepackage{amsmath}
\usepackage{graphicx}
\usepackage{amssymb}
\usepackage{bbm}
\usepackage{amsthm,mathtools}
\usepackage{ulem}
\usepackage{geometry}
\usepackage{tikz-cd}
\usepackage{mathrsfs}
\usepackage[colorinlistoftodos]{todonotes}
\usepackage{enumitem}
\usepackage{verbatim}
\usepackage[foot]{amsaddr}
\usepackage{dsfont}

\makeatletter

\@namedef{subjclassname@2010}{
	
	\textup{2020} Mathematics Subject Classification}

\makeatother
\newtheorem{thm}{Theorem}[section]
\newtheorem*{thm*}{Theorem}

\newtheorem{cor}{Corollary}

\newtheorem{lem}[thm]{Lemma}

\theoremstyle{definition}

\numberwithin{equation}{section}

\newcommand{\mbr}{\mathbb{R}}

\newcommand{\mcs}{\mathcal{S}}

\newcommand{\mme}{\mathrm{e}}

\newcommand{\M}{{\mathcal M}}

\newcommand{\F}{\mathcal F}

\usepackage{hyperref}
\hypersetup{hypertex=true,colorlinks=true,linkcolor=blue,anchorcolor=blue,citecolor=blue}
\newcommand{\newabstract}[1]{%
	\par\bigskip
	\csname otherlanguage*\endcsname{#1}%
	\csname captions#1\endcsname
	\item[\hskip\labelsep\scshape\abstractname.]
}

\begin{document}

	\baselineskip=17pt

	\title[Large  quadratic character sums]{Large  quadratic character sums}

	\author{Zikang Dong}
	\author{Yanbin Zhang}
	\address[Zikang Dong]{School of Mathematical Sciences, Soochow University, Suzhou 215006, P. R. China}
	\address[Yanbin Zhang]{School of Mathematics, Shandong University, Jinan 250100, P. R. China}
	\email{zikangdong@gmail.com}
	\email{x15yanbin.zhang@gmail.com}
	
	\date{}
	
	\begin{abstract} 
		In this article, we investigate conditional large values of quadratic Dirichlet character sums. We prove some Omega results of quadratic character sums under the assumption of the generalized Riemnn hypothesis, which are as sharp as previous results for all characters  modulo a large prime.
	\end{abstract}

	
	\maketitle
	
	\section{Introduction}
For any large prime number $q$ and any character $\chi({\rm mod}\;q)$, we always have the P\'olya-Vinogradov inequality for any $x>0$
   $$\sum_{n\le x}\chi(n)\ll\sqrt q\log Q,$$
   where $Q=q$ unconditionally and $Q=\log q$ on the generalized Riemann Hypothesis (GRH). The conditional upper bound is optimal up to the implied constant, since Paley showed for any large $q$, there always exists a quadratic character $\chi$ and $x<q$ such that
   $$\sum_{n\le x}\chi(n)\gg\sqrt q\log_2 q.$$
   We use $\log_j$ to denote the $j$-th iteration of logarithm. When $x$ is fixed, consider the maximum
   $$\Delta_q(x):=\max_{\chi_0\neq\chi({\rm mod }\;q)}\Big|\sum_{n\le x}\chi(n)\Big|.$$
   Granville and Soundararajan \cite{GS01} conjectured $S_q(x)$ increases on $x<\sqrt q$. They showed several evidence to believe this. That is, lower bounds for $S_q(x)$ according to the range of $x$. See Theorems 3--8 of \cite{GS01}.  Most of the results are divided by comparing the size of $x$ with the quantity $\exp(\sqrt{\log q}).$ When $x\le\exp((\log q)^{\frac12-\varepsilon})$ we say the character sum $\sum_{n\le x}\chi(n)$ is `short'. When $x\ge\exp((\log q)^{\frac12+\varepsilon})$ we may call it `long'.

   Part of these have been improved separately by Munsch \cite{Mun}, Hough \cite{Hou} and La Bret\`eche and Tenenbaum \cite {BT}. When $\log q\le x\le \exp(\sqrt{\log q})$, Munsch \cite{Mun} showed that
    $$ \Delta_q(x)\ge \Psi\bigg(x,\big(\tfrac14+o(1)\big)\frac{\log q\log_2q}{\max\{\log_2x-\log_3q,\log_3q\}}\bigg).$$
    When $x=\exp\left(\tau\sqrt{\log q\log_2q}\right)$, Hough \cite{Hou} showed that
   \[ \Delta_q(x)\ge \sqrt{x}\exp\bigg((1+o(1))A(\tau+\tau')\sqrt{\frac{\log X}{\log_2 X}}\bigg),\]
 where $A,\tau,\tau'\in\mbr$ such that $\tau=(\log_2q)^{O(1)}$ and
   \[\tau=\int_A^\infty \frac{\mme^{-u}}{u}d u,\qquad \tau'=\int_A^\infty\frac{\mme^{-u}}{u^2}d u.\]
   When $\exp((\log q)^{\frac12+\delta})\le x\le q$, La Bret\`eche and Tenenbaum \cite {BT} showed that
$$ \Delta_q(x)\ge \sqrt x\exp\bigg((\sqrt2+o(1))\sqrt{\frac{\log( q/x)\log_3(q/x)}{\log_2(q/x)}}\bigg).$$

   As probably the most special characters, the real primitive characters play a important role. For example, Paley's lower bound was based on the study of real primitive characters.  Granville and Soundararajan \cite{GS01} also showed some results for lower bounds of real primitive character sums which are analogous to the results of $\Delta_q(x)$. See Theorems 9--11 in \cite{GS01}. Most of these results have not been improved due to bad order of error terms in the mean values of quadratic character sums $\sum_{d\le X,\, d\in\F}\chi_d(n)$. Here we denote by $\F$  the set of all fundamental discriminants. The aim of this article is to study the lower bounds for the following  quantity
   $$\max_{X<d\le 2X\atop d\in\F}\sum_{n\le x}\chi_d(n),$$
   and improve Granville and Soundararajan's results, at the cost of assuming GRH. More precisely, we establish analogous results for real primitive characters to these of Munsch \cite{Mun}, Hough \cite{Hou} and La Bret\`eche and Tenenbaum \cite {BT}. Note that similar results were also established for zeta sums, see \cite{DWZ}. For the distribution of large quadratic character sums, we refer to \cite{Lam24} and \cite{DWZ23}. For moments of quadratic character sums, we refer to Munsch's recent work \cite{Mun25}.
   
   For  short quadratic character sums, we have the following lower bounds, which are analogous (but weaker) to the results of Munsch \cite{Mun}.
\begin{thm}\label{thm1.1}
Assume GRH. Let $\log X\le x\le\exp((\log X)^{\frac12})$, then we have 
    $$ \max_{X<d\le 2X\atop d\in\F}\sum_{n\le x}\chi_d(n)\ge \Psi\bigg(x,\big(\tfrac14+o(1)\big)\frac{\log X\log_2X}{\max\{\log_2x-\log_3X,\log_3X\}}\bigg).$$
\end{thm}

When $\log x$ is a small power of $\log X$, we can write the lower bound in a more compact way.

\begin{cor}
Assume GRH. Let $\log x=(\log X)^\sigma$ for a fixed $0<\sigma<1/2$. Then we have
$$ \max_{X<d\le 2X\atop d\in\F}\sum_{n\le x}\chi_d(n)\ge \Psi\Big(x,\big(\tfrac{1}{2\sigma}+o(1)\big)\log X\Big).$$
\end{cor}
 When $x$ is even smaller (power of $\log X$), we can write the lower bound more precisely. 
\begin{cor} Assume GRH.
Let $x=(\log X)^A$ for some  $A>1$. Then we have
$$\max_{X<d\le 2X\atop d\in\F}\sum_{n\le x}\chi_d(n)\ge \Psi\bigg(x,\big(\tfrac{1}{2}+o(1)\big)\frac{\log X\log_2X}{\log_3X} \bigg).$$
\end{cor}

When $x$ is around $\exp((\log X)^{\frac12})$, we have the following result analogous to Theorem 3.2 of Hough \cite{Hou}.
\begin{thm}\label{thm1.2}
   Assume GRH. Let $x=\exp\left(4\sqrt{\log X\log_2X}\log_3X\right)$.
   %\[\tau=\int_A^\infty \frac{\mme^{-u}}{u}d u,\qquad \tau'=\int_A^\infty\frac{\mme^{-u}}{u^2}d u.\]
Then we have 
   \[ \max_{X<d\le 2X\atop d\in\F}\sum_{n\le x}\chi_d(n)\ge \sqrt{x}\exp\bigg((1+o(1))\sqrt{\frac{\log X}{\log_2 X}}\bigg).\]
\end{thm}

Finally, when the sum is long, we have the result analogous to Theorem 1.6 of \cite{BT}. Although it is slightly weaker than their result.
\begin{thm}\label{thm1.3}
    Assume GRH. Let $\exp((\log X)^{\frac12+\varepsilon})<x\le   X^{\frac12}$, then we have 
    $$ \max_{X<d\le 2X\atop d\in\F}\sum_{n\le x}\chi_d(n)\ge \sqrt x\exp\bigg((1+o(1))\sqrt{\frac{\log(\sqrt X/x)\log_3(\sqrt X/x)}{\log_2(\sqrt X/x)}}\bigg).$$
\end{thm}

The main method to establish these conditional upper bounds is the resonance method, which was highly developed by Hilberdink \cite{Hil} and Soundararajan \cite{Sound}. Another important technique is Lemma \ref{lem2.2} established by Darbar and Maiti \cite{DM}, which gives a good evaluation for the mean values of quadratic characters under GRH. This is much stronger than the unconditional result of Granville and Soundararajan \cite{GS}.

This paper is arranged as follows. We present some lemmas in \S \ref{sec2}. We prove Theorems \ref{thm1.1}--\ref{thm1.3} separately in \S \ref{sec3}--\ref{sec5}.

    \section{Preliminary Lemmas}\label{sec2}
    Firstly, when $n$ is a fixed integer, we have the following mean-value result unconditionally.
    \begin{lem}\label{lem2.1}
	If $n=\square$ we have
	\begin{align*}
	\sum_{|d|\le X\atop d\in\F} \chi_{d}(n)=\frac{X}{\zeta(2)}\prod_{p|n}\frac{p}{p+1}+O\Big(X^{\frac12}\tau(\sqrt n)\Big) .
	\end{align*}
	If $n\neq\square$ we have
    	\begin{align*}
	\sum_{|d|\le X\atop d\in\F} \chi_{d}(n)=O\left(X^{\frac12}n^{\frac14}\log n\right).
	\end{align*}
\end{lem}
\begin{proof}
    This is Lemma 4.1 of \cite{GS}.
\end{proof}
Under GRH, the error terms above can be improved much better.
    \begin{lem}\label{lem2.2}
	Assuming GRH. Let $n=n_0n_1^2$ be a positive integer with $n_0$ the square-free part of $n$.
	%Let $\beta\geq 2$ and $\frac{1}{2}<\sigma< 1$ with $\sigma>\frac{\beta -1}{\beta}$.  
	 Then for any $\varepsilon>0$, we obtain
	\begin{align*}
	\sum_{|d|\le X\atop d\in\F} \chi_{d}(n)=\frac{X}{\zeta(2)}\prod_{p|n}\frac{p}{p+1}\mathds{1}_{n=\square}+ O\left(X^{\frac12+\varepsilon}f(n_0)g(n_1)\right),
	\end{align*}
	where   ${\mathds{1}}_{n=\square}$ indicates the indicator function of the square numbers, and
    $$f(n_0)=\exp((\log n_0)^{1-\varepsilon}),\;\;\;\;g(n_1)=\sum_{d|n_1}\frac{\mu(d)^2}{d^{\frac12+\varepsilon}}.$$
\end{lem}
\begin{proof}
    This follows directly from Lemma 1 of \cite{DM}.
\end{proof}
On the one hand, it is clear that 
$$f(n_0)\le n_0^\varepsilon\le n^\varepsilon,\;\;\;\;g(n_1)\le n_1^\varepsilon\le n^\varepsilon.$$
On the other hand, if we denote the largest prime factor of $n$ by $P_+(n)$, then $n_0,n_1\le \prod_{p\le P_+(n)}p$.
So easily we have
$$f(n_0)\le\exp\big(P_+(n)^{1-\varepsilon}\big),\;\;\;\;\;g(n_1)\le\exp\big(P_+(n)^{\frac12-\varepsilon}\big).$$

The following lemma plays a key role in the proof of Theorem \ref{thm1.2}.
\begin{lem}\label{rmrn}
    Let $Y$ be large and $\lambda=\sqrt{\log Y\log_2 Y}$. Define the multiplicative function $r$ supported on square-free integers and for any prime $p$:
$$r(p)=\begin{cases}
   \frac{\lambda}{\sqrt p \log p}, &  \lambda\le p\le \exp((\log\lambda)^2),\\
   0, & {\rm otherwise.}
\end{cases}$$
If $\log N>3\lambda\log_2\lambda$, then we have
\begin{equation}\label{DD}
    \sum_{a,b\le Y}\sum_{m,n\le N\atop an=bm}r(a)r(b)\Big/\sum_{n\le Y}r(n)^2\ge N\exp{\bigg((2+o(1))\sqrt{\frac{\log Y}{\log_2 Y}}}\bigg).
\end{equation}
\end{lem}
\begin{proof}
    This follows directly from Page 97 of \cite{Hou}.
\end{proof}

We also need the following result for GCD sums. The reltion between extreme values of arithmetic functions and GCD sum was firstly discovered by Aistleitner \cite{Ais}.
\begin{lem}\label{GCD}
    Let $\M$ be any set of positive squarefree integers with $|\M|=N$. Then as $N\to\infty$, we have
    $$\max_{|\M|=N}\sum_{m,n\in\M}\sqrt{\frac{(m,n)}{[m,n]}}=N\exp\bigg((2+o(1))\sqrt{\frac{\log N\log_3N}{\log_2N}}\bigg).$$
\end{lem}
\begin{proof}
    This is Eq. (1.5) of \cite{BT}.
\end{proof}
Note that in the proof of the above lemma, the choice for the set $\M$ satisfies $y_\M:=\max_{m\in\M}P_+(m)\le (\log N)^{1+o(1)}$.
	%%%%%%%%%%%%%%%%%%%%%%%%%%%%%%%%%%%%%%%%%%%%%%%%%%%%%%%%%%%%%%%%%%%%%%%%%
	\section{Proof of Theorem \ref{thm1.1}}\label{sec3}

 Let $\alpha$ be a very small positive number. For $$y:=\Big(\frac14-\alpha\Big)\frac{\log X\log_2X}{\max\{\log_2x-\log_3X,\log_3X\}},$$ let $a_k$ be completely multiplicative with $a_1=1$, $a_p=1-\frac{\log y}{\log x(\log_2X)^{1+\delta}}$ for $p\le y$ and $a_p=0$ for $p>y$. Here $\delta$ is a positive number smaller than $\alpha$.
We have the following result for $a_k$, which follows directly from \cite[pp. 35-36]{Mun}.
\begin{lem}\label{lem5.1}Let $a_k$ and $y$ be defined above. We have
    $$\sum_{k\le x\atop k\in\mcs(y)} a_k\ge\Psi(x,(1+o(1))y).$$
\end{lem}
Define the resonators
$$R(d):=\prod_{p\le y}\big(1-{a_p}\chi_d(p)\big)^{-1}=\sum_{  k\in\mcs(y)}{a_k}{\chi_d(k)},$$
where $a_k$ is defined above. We have
$$\log R(d)\le-\sum_{p\le y}\log(1-a_p)\le\big(\tfrac14-\tfrac\alpha2\big)\log X.$$
So 
$$R(d)^2\le X^{\frac12-\alpha}.$$
Define
$$M_1(R,X):=\sum_{X<d\le2X}R(d)^2,$$
and 
$$M_2(R,X):=\sum_{X<d\le2X}S_d(x)R(d)^2,$$
where $S_d(x):=\sum_{n\le x}\chi_d(n)$. 
Then 
$$\max_{X< d\le 2X\atop d\in\F}S_d(x)\ge\frac{M_2(R,X)}{M_1(R,X)}.$$
For $M_1(R,X)$ we have
$$M_1(R,X)=\sum_{m,n\in S(y)}a_ma_n\sum_{X<d\le 2X\atop d\in\F}\chi_d(mn).$$
By Lemma \ref{lem2.2}, we have
\begin{align*}M_1(R,X)&=\frac{X}{\zeta(2)}\sum_{m,n\in S(y)\atop mn=\square}a_ma_n\prod_{p|mn}\frac{p}{p+1}+O\Big(X^{\frac12+\varepsilon}\exp(y^{1-\varepsilon})\sum_{m,n\in S(y)}a_ma_n \Big)\\
&=\frac{X}{\zeta(2)}\sum_{m,n\in S(y)\atop mn=\square}a_ma_n\prod_{p|mn}\frac{p}{p+1}+O\Big(X^{\frac12+\varepsilon}\Big(\sum_{m\in S(y)}a_m\Big)^2 \Big).\end{align*}
For the sum in the O-term we have
$$\sum_{m\in S(y)}a_m=\prod_{p\le y}(1-a_p)^{-1}=\prod_{p\le y}\frac{\log x(\log_2X)^{1+\delta}}{\log y}=\Big(\frac{\log x(\log_2X)^{1+\delta}}{\log y}\Big)^{\pi(y)}\ll X^{\frac14-\frac\alpha2}.$$
So $$M_1(R,X)=\frac{X}{\zeta(2)}\sum_{m,n\in S(y)\atop mn=\square}a_ma_n\prod_{p|mn}\frac{p}{p+1}+O\Big(X^{1-\alpha+\varepsilon}\Big).$$

Similarly, for $M_2(R,X)$, we have 
\begin{align*}&M_2(R,X)\\&=\sum_{k\le x}\sum_{m,n\in S(y)}a_ma_n\sum_{X<d\le 2X\atop d\in\F}\chi_d(kmn)\\&=\frac{X}{\zeta(2)}\sum_{k\le x}\sum_{m,n\in S(y)\atop kmn=\square}a_ma_n\prod_{p|kmn}\frac{p}{p+1}+O\Big(X^{\frac12+\varepsilon}\exp(y^{1-\varepsilon})x^\varepsilon\sum_{k\le x}\sum_{m,n\in S(y)}a_ma_n  \Big)
\\&=\frac{X}{\zeta(2)}\sum_{k\le x}\sum_{m,n\in S(y)\atop kmn=\square}a_ma_n\prod_{p|kmn}\frac{p}{p+1}+O\Big(X^{\frac12+\varepsilon}\exp(y^{1-\varepsilon})x^{1+\varepsilon}\Big(\sum_{m\in S(y)}a_m\Big)^2 \Big)
\\&=\frac{X}{\zeta(2)}\sum_{k\le x}\sum_{m,n\in S(y)\atop kmn=\square}a_ma_n\prod_{p|kmn}\frac{p}{p+1}+O\Big(X^{\frac12+\varepsilon}\Big(\sum_{m\in S(y)}a_m\Big)^2 \Big)\\
&=\frac{X}{\zeta(2)}\sum_{k\le x}\sum_{m,n\in S(y)\atop kmn=\square}a_ma_n\prod_{p|kmn}\frac{p}{p+1}+O\Big(X^{1-\alpha+\varepsilon}\Big).\end{align*}

Let $$I_1(R,X):=\sum_{m,n\in S(y)\atop mn=\square}a_ma_n\prod_{p|mn}\frac{p}{p+1},$$
and $$I_2(R,X):=\sum_{k\le x}\sum_{m,n\in S(y)\atop kmn=\square}a_ma_n\prod_{p|kmn}\frac{p}{p+1}.$$
Then $$\frac{M_2(R,T)}{M_1(R,T)}=\frac{I_2(R,T)}{I_1(R,T)}+O(X^{-\alpha+\varepsilon}).$$
For $I_2(R,X)$, we have
\begin{align*}
I_2(R,X)&=\sum_{k\le x}\sum_{m,n\in\mcs(y)\atop kmn=\square}a_ma_n\prod_{p|kmn}\frac{p}{p+1}\\
&\ge\sum_{k\in\mcs(y)\atop k\le x}\sum_{m,n\in\mcs(y)\atop kmn=\square,k|m}a_ma_n\prod_{p|kmn}\frac{p}{p+1}\\
&=\sum_{k\in\mcs(y)\atop k\le x}\sum_{\ell,n\in\mcs(y)\atop \ell n=\square}a_{k\ell}a_n\prod_{p|k\ell n}\frac{p}{p+1}\\
&\ge\sum_{k\in\mcs(y)\atop k\le x}a_k\prod_{p|k}\frac{p}{p+1}\sum_{\ell,n\in\mcs(y)\atop \ell n=\square}a_{\ell}a_n\prod_{p|\ell n}\frac{p}{p+1}\\
&=I_1(R,X)\sum_{k\in\mcs(y)\atop k\le x} a_k\prod_{p|k}\frac{p}{p+1} .
\end{align*}
So we deduce that
$$\frac{I_2(R,X)}{I_1(R,X)}\ge\sum_{k\le x\atop k\in\mcs(y)} a_k\prod_{p|k}\frac{p}{p+1}\ge\Psi(x,(1+o(1))y),$$
by Lemma \ref{lem5.1}, which completes the proof.

\section{Proof of Theorem \ref{thm1.2}}\label{sec4}

Let $y=X^{\frac12-\delta}/x^2$ and  $\lambda=\sqrt{\log y\log_2y}$, where $0<\delta<\frac{1}{100}$ is any fixed small number. We define a completely multiplicative function (supported on square-free numbers) $r(n)$ by $r(p)=\frac{a}{\sqrt{p}\log p}$ where $\lambda^2\leq p\leq \mme^{(\log \lambda)^2}$ is prime and $r(p)=0$ for other primes.  We define the resonator $$R(d):=\sum_{n\leq y}r(n)\chi_d(n),$$and 
$$M_1(R,X):=\sum_{X<d\le2X\atop d\in\F}R(d)^2,$$
$$M_2(R,X):=\sum_{X<d\le2X\atop d\in\F}R(d)^2S_d(x)^2.$$
Then we have 
\[\max_{X<d\le2X\atop d\in\F}\sum_{n\le x}\chi_d(n)\geq \sqrt{\frac{M_2(R,X)}{M_1(R,X)}}.\]
By Lemma \ref{lem2.2}, we have 
\begin{align}\label{eq:rtdt}
M_1(R,X)&=\frac{X}{\zeta(2)}\sum_{m,n\le y\atop mn=\square}r(m)r(n)\prod_{p|mn}\frac{p}{p+1}+O\Big(X^{\frac12+\varepsilon}\sum_{m,n\le y}r(m)r(n)\Big)\\\nonumber
&\le\frac{X}{\zeta(2)}\sum_{m\le y}r(m)^2+O\Big(X^{\frac12+\varepsilon}y\sum_{m\le y}r(m)^2\Big)\\
&=\frac{X}{\zeta(2)}\sum_{m\le y}r(m)^2+O\Big(X^{1-\delta+\varepsilon}\sum_{m\le y}r(m)^2\Big),
\end{align}
and 
\begin{align}\label{eq:strtdt}
  M_2(R,X)&=\frac{X}{\zeta(2)}\sum_{\substack{ m,n\leq y\\ k,\ell\leq x\\ k\ell mn=\square}}r(m)r(n)\prod_{p|mnk\ell}\frac{p}{p+1}+O\Big(X^{\frac12+\varepsilon}\sum_{ m,n\leq y\atop k,\ell\leq x}r(m)r(n)\Big)
  \\&=\frac{X}{\zeta(2)}\sum_{\substack{ m,n\leq y\\ k,\ell\leq x\\ k\ell mn=\square}}r(m)r(n)\prod_{p|mnk\ell}\frac{p}{p+1}+O\Big(X^{\frac12+\varepsilon}x^2y\sum_{ m\leq y}r(m)^2\Big)\\
  &\ge\frac{X}{\zeta(2)}\sum_{\substack{ m,n\leq y\\ k,\ell\leq x\\ mk=n\ell}}r(m)r(n)\prod_{p|mk}\frac{p}{p+1}+O\Big(X^{1-\delta+\varepsilon}\sum_{ m\leq y}r(m)^2\Big),
\end{align}
where we used $y=X^{\frac12-\delta}/x^2$.
Now by combining Eq. \eqref{eq:rtdt} and Eq. \eqref{eq:strtdt} we get 
\begin{align*}
   \Big(\max_{X<d\le2X\atop d\in\F}\sum_{n\le x}\chi_d(n)\Big)^2
   &\geq {\sum_{\substack{ m,n\leq y\\ k,\ell\leq x\\ mk=n\ell}}r(m)r(n)\prod_{p|mk}\frac{p}{p+1}}\Big/{\sum_{m\leq y}r(m)^2}+O(X^{-\delta+\varepsilon})\\ 
   &\geq(\log X)^{-c}\sum_{\substack{ m,n\leq y\\ k,\ell\leq x\\ mk=n\ell}}r(m)r(n)\Big/{\sum_{m\leq y}r(m)^2}+O(X^{-\delta+\varepsilon})
   ,
\end{align*}
where we used 
$$\prod_{p|mk}\frac{p}{p+1}\ge\prod_{p\le X}\frac{p}{p+1}\ge(\log X)^{-c}$$
for some absolute positive $c$.
Finally, Theorem \ref{thm1.2} follows from Lemma \ref{DD}.

\begin{comment}
Now let $d_1=(m,n),d_2=(k,\ell)$, $m_1=m/d_1,n_1=n/d_1$. Then we have 
\begin{align*}
&\sum_{\substack{k,\ell\leq x\\ m,n\leq y\\ mk=n\ell}}r(m)r(n)=\sum_{m_1,n_1\leq y\atop (m_1,n_1)=1}r(m_1)r(m_2)\sum_{\substack{d_1\leq y/\max(m_1,n_1)\\ d_2<x/\max(m_1,n_1)\\ (d_1,m_1n_1)=1}}r(d_1)^2\\
=&\sum_{m_1,n_1\leq x\atop (m_1,n_1)=1}r(m_1)r(m_2)\sum_{\substack{d_1\leq y/\max(m_1,n_1)\\(d_1,m_1n_1)=1}}r(d_1)^2\left(x/\max(m_1,n_1)+O(1)\right).
\end{align*}
The last equality holds since there is no positive integer $d_2\leq x/\max(m_1,n_1)$ when $\max(m_1,n_1)>x$. Then the proof in \cite[P. 97]{Hough} gives 
\begin{align*}
    \max_{t\in[1,T]}\Big|\sum_{n\le x}n^{{\rm i}t}\Big|^2&\geq \sum_{m_1,n_1\leq x\atop (m_1,n_1)=1}\frac{xr(m_1)r(m_2)}{\max(m_1,n_1)}\sum_{\substack{d_1\leq y/\max(m_1,n_1)\\(d_1,m_1n_1)=1}}\frac{r(d_1)^2}{\prod_p(1+r(p)^2}+O\left(\frac{x^4}{T}\right) \\
    &\geq x\exp\left((2+o(1))\sqrt{\frac{\log y}{\log_2 y}}\right)+O\left(\frac{x^4}{T}\right)\\
    &=x\exp\left((2+o(1))\sqrt{\frac{(1-\theta)\log T}{\log_2 T}}\right)
\end{align*}
The last equality follows from our assumption that $x=T^\theta<T^{\frac{1}{3}}$.
\end{comment}

\section{Proof of Theorem \ref{thm1.3}}\label{sec5}

Let $\M$  be a set of positive squarefree integers satisfying the conditions in Lemma \ref{GCD}, with cardinality $|\M|=N=\lfloor X^{\frac12-\delta}/x\rfloor$. Recall that $y_\M=\max_{m\in\M}P_+(m)\le(\log N)^{1+o(1)}$. Here $\delta<\frac12$ is any small constant. Define the resonator
$$
R(d) := \sum_{m\in\M} \chi_d(m),
$$
and
define
$$M_1(R,X):=\sum_{X<d\le2X\atop d\in\F}R(d)^2,$$
and 
$$M_2(R,X):=\sum_{X<d\le2X\atop d\in\F}S_d(x)^2R(d)^2.$$
Then
\begin{align}\max_{X<d\le2X}S_d(x)^2\ge\frac{M_2(R,X)}{M_1(R,X)}.\label{extreme}\end{align}
For $M_1(R,X)$, it holds that 
\begin{align*}M_1(R,X)&=\frac{X}{\zeta(2)}\sum_{m,n\in\M\atop mn=\square}\prod_{p|mn}\frac{p}{p+1}+O\Big(X^{\frac12+\varepsilon}\sum_{m,n\in\M}1\Big)\\
&=\frac{X}{\zeta(2)}\sum_{m\in\M}+O\big(X^{\frac12+\varepsilon}N^2\big)\\
&\le\frac{X}{\zeta(2)}N\prod_{p\le X}\frac{p}{p+1}+O\big(X^{\frac12+\varepsilon}N^2\big)\\
&\le\frac{X}{\zeta(2)}N+O\big(X^{\frac12+\varepsilon}N^2\big)\\
&\ll\frac{X}{\zeta(2)}N.\end{align*}
For $M_2(R,X)$, we have
\begin{align}
M_2(R,X)&=\frac{X}{\zeta(2)}\sum_{m,n\in\M}\sum_{k,\ell\le x\atop mnk\ell=\square}\prod_{p|mnk\ell}\frac{p}{p+1}+O\Big(X^{\frac12+\varepsilon}\sum_{m,n\in\M}\sum_{k,\ell\le x}1\Big)\nonumber\\
&=\frac{X}{\zeta(2)}\sum_{m,n\in\M}\sum_{k,\ell\le x\atop mnk\ell=\square}\prod_{p|mnk\ell}\frac{p}{p+1}+O\big(X^{\frac12+\varepsilon}N^2x^2\big).\label{mainterm}
\end{align}
So
$$\frac{M_2(R,X)}{M_1(R,X)}\gg\frac{1}{N}\sum_{m,n\in\M}\sum_{k,\ell\le x\atop mnk\ell=\square}\prod_{p|mnk\ell}\frac{p}{p+1}+O(X^{-\delta+\varepsilon}x).$$
Let
$$I_2(R,X):=\sum_{m,n\in\M}\sum_{k,\ell\le x\atop mnk\ell=\square}\prod_{p|mnk\ell}\frac{p}{p+1}.$$
 We have 
$$I_2(R,X)\ge\sum_{m,n\in\M}\sum_{k,l\le x \atop mk=n\ell}\prod_{p|mk}\frac{p}{p+1}
\ge\prod_{p\le X}\frac{p}{p+1}\sum_{m,n\in\M}\sum_{k,l\le x \atop mk=n\ell}1
\ge(\log X)^{-c}\sum_{m,n\in\M}\sum_{k,l\le x \atop mk=n\ell}1,$$
for some $c>0$.
For fixed $m,n$, $mk=n\ell$ implies $k=nL/(m,n)$ and $\ell=mL/(m,n)$ for some integer $L$. Since $\max\M\le2\min\M$, we have for the inner sum
\begin{align*}\sum_{k,\ell\le x\atop mk=n\ell}1
&\ge\frac{x}{\max\{\frac{m}{(m,n)},\frac{n}{(m,n)}\}}\\
&\gg\frac{x}{\sqrt{2\frac{m}{(m,n)}\frac{n}{(m,n)}}}\\
&\gg x\sqrt{\frac{(m,n)}{[m,n]}}.\end{align*}
It follows that
$$I_2(R,X)\gg x(\log X)^{-c}\sum_{m,n\in\M}\sqrt{\frac{(m,n)}{[m,n]}}.$$
Inserting into \eqref{mainterm} and back to \eqref{extreme}, we have 
\begin{align*}
\max_{X<d\le2X}S_d(x)^2&\gg \frac x {N}(\log X)^{-c}\sum_{m,n\in\M}\sqrt{\frac{(m,n)}{[m,n]}}+O(X^{-\delta+\varepsilon}x)\\
&\gg  x{(\log X)^{-c}}\exp\bigg((2+o(1))\sqrt{\frac{\log N\log_3N}{\log_2N}}\bigg)\\
&\ge x\exp\bigg((2+o(1))\sqrt{\frac{\log (X^{\frac12-\delta}/x)\log_3(X^{\frac12-\delta}/x)}{\log_2(X^{\frac12-\delta}/x)}}\bigg),
\end{align*}
where we have used Lemma \ref{GCD}. Thus we complete the proof of Theorem \ref{thm1.3}, since $\delta$ can be arbitraly small.
	\section*{Acknowledgements}
	The authors are supported by the Shanghai Magnolia Talent Plan Pujiang Project (Grant No. 24PJD140) and the National
	Natural Science Foundation of China (Grant No. 	1240011770).

	\normalem

\end{document}